\documentclass[11pt]{article}

\usepackage[T1]{fontenc}
\usepackage{lmodern}
\usepackage[margin=1.05in]{geometry}
\usepackage{amsmath,amssymb,amsthm,mathtools,bm}
\usepackage{booktabs,tabularx,array}
\usepackage{enumitem}
\usepackage{microtype}
\usepackage{xcolor}
\usepackage[colorlinks=true,linkcolor=blue!55!black,citecolor=blue!55!black,urlcolor=blue!55!black]{hyperref}
\usepackage[nameinlink,noabbrev]{cleveref}
\hypersetup{
  pdftitle={Bulk-edge sticking beyond the Perron mode in Gaussian softmax attention},
  pdfauthor={Alexander Jerschow}
}

\allowdisplaybreaks

\newtheorem{theorem}{Theorem}[section]
\newtheorem{proposition}[theorem]{Proposition}
\newtheorem{lemma}[theorem]{Lemma}
\newtheorem{corollary}[theorem]{Corollary}
\theoremstyle{definition}

\theoremstyle{remark}
\newtheorem{remark}[theorem]{Remark}

\crefname{theorem}{Theorem}{Theorems}
\Crefname{theorem}{Theorem}{Theorems}
\crefname{proposition}{Proposition}{Propositions}
\Crefname{proposition}{Proposition}{Propositions}
\crefname{lemma}{Lemma}{Lemmas}
\Crefname{lemma}{Lemma}{Lemmas}
\crefname{corollary}{Corollary}{Corollaries}
\Crefname{corollary}{Corollary}{Corollaries}
\crefname{definition}{Definition}{Definitions}
\Crefname{definition}{Definition}{Definitions}
\crefname{remark}{Remark}{Remarks}
\Crefname{remark}{Remark}{Remarks}

\newcommand{\R}{\mathbb R}
\newcommand{\E}{\mathbb E}
\newcommand{\Pp}{\mathbb P}
\newcommand{\one}{\mathbf 1}

\newcommand{\Tr}{\operatorname{Tr}}
\newcommand{\diag}{\operatorname{diag}}
\newcommand{\supp}{\operatorname{supp}}
\newcommand{\MP}{\mathrm{MP}}

\newcommand{\norm}[1]{\lVert #1\rVert}
\newcommand{\lmaxabs}[1]{\lambda_{\max}\!\left(\left|#1\right|\right)}
\newcommand{\abs}[1]{\lvert #1\rvert}
\newcommand{\Kbeta}{\mathcal K_\beta}
\newcommand{\ESD}{\operatorname{ESD}}

\title{Bulk-edge sticking beyond the Perron mode\\
in Gaussian softmax attention}
\author{Alexander Jerschow\\
\small Graduate School of Mathematics, Nagoya University\\
\small \href{mailto:jerschowa02@gmail.com}{\texttt{jerschowa02@gmail.com}}}
\date{September 2026}

\begin{document}
\maketitle

\begin{abstract}
This paper studies row-softmax self-attention with independent Gaussian query
and key weights in the proportional regime, at fixed inverse temperature.
The Gaussian-equivalence result in \cite{HayaseCollinsKarakida2026} identifies
the empirical squared singular-value distribution after removal of the Perron
direction.  A global law alone does not exclude finitely many nonleading
outliers.  It is proved here that no such outliers persist: the rescaled
squared singular value $\ell s_k(A)^2$ converges in probability to the upper
edge of that bulk law for every fixed $k\geq2$.  In fact, this convergence is
uniform over any deterministic sublinear number of leading non-Perron
indices.  The proof uses an exact decomposition of the softmax normalization,
conditions on the key matrix, identifies the conditional covariance exactly
with a diagonally conjugated inner-product kernel, linearizes that kernel in
operator norm, and applies the outside-support local law in
\cite{FanMaPaquetteWang2026}.  A separate stability argument identifies the
finite conditional deformed Marchenko--Pastur edge with the limiting bulk
edge.  A scalar formula for that edge is also derived throughout the
proportional regime, recovering the explicit square-model formula in
\cite{HayaseCollinsKarakida2026}, including its physical branch.
\end{abstract}



\section{Introduction}

Let $A$ be a random row-stochastic attention matrix.  Its deterministic right
eigenvector $\one$ identifies the Perron direction and yields one macroscopic
singular value.  The
remaining singular values are much smaller.  In the fixed-temperature
Gaussian model, the empirical law of the squared singular values of
\(
  \sqrt\ell(A-u_\ell u_\ell^\top)
\)
converges to a compactly supported deterministic measure
$\nu_\infty$~\cite[Theorem~3.2]{HayaseCollinsKarakida2026}.  This result
determines the bulk scale and bulk shape, but it does not determine the top
few non-Perron singular values: changing finitely many eigenvalues does not
change an empirical spectral distribution.

This paper studies the upper spectral edge in the Gaussian,
row-isometric-input model.  Here and below, ``non-Perron singular values''
means the nonleading singular values $s_k(A)$, $k\geq2$, lying beyond the
single macroscopic mode induced by the exact Perron eigenvector.  It is proved
below that all fixed such singular values stick to the upper bulk edge.  If
\(
  \ell/d\to\gamma\in(0,1]
\)
and
\(
  d_{qk}/d\to\psi\in(0,\infty)
\), then, for each fixed $k\geq2$,
\[
  \ell s_k(A)^2\xrightarrow{\Pp}
  E_+(\beta,\gamma/\psi):=\max\supp\nu_\infty(\beta,\gamma,\psi).
\]
Equivalently,
\(
  d s_k(A)^2\xrightarrow{\Pp}E_+/\gamma
\).

There are three points in the proof that do not follow from the global law.
First, the operator attached directly to the non-Perron singular values is a
double compression, not only the right-centered matrix used for the empirical
law.  Second, after conditioning on the keys, the nonlinear feature rows are
independent but have dependent coordinates; an independent-entry edge theorem
does not apply.  Third, the finite-dimensional, conditional deformed
Marchenko--Pastur edge must itself be shown to converge to the limiting bulk edge.  These points
are handled respectively by Cauchy interlacing, the non-separable
sample-covariance local law in \cite{FanMaPaquetteWang2026}, and a direct
analysis of the inverse Stieltjes map.

The mechanism is also useful conceptually.  Conditioning produces two
low-rank nonlinear structures: a conditional-mean direction and a quadratic
kernel direction.  The left Perron projection removes the former exactly,
while the double compression removes the latter to leading order.
Related conditional-centering and quadratic-equivalent mechanisms occur in
the nonlinear random-feature literature
\cite{PenningtonWorah2017,FanWang2020,BenigniPeche2021,
BenigniPeche2022,CranstonWangKempMahoney2026}.  Here they are combined with
the exact row normalization and the singular-value geometry of attention.

\paragraph{Organization.}
\Cref{sec:model} states the model and main theorem.
\Cref{sec:edge-description} describes the limiting edge.
\Cref{sec:softmax,sec:normalizer} treat the exact normalization.
\Cref{sec:conditioning,sec:kernel} identify and approximate the conditional
covariance.  \Cref{sec:local-law,sec:finite-edge} supply the upper spectral
bound.  \Cref{sec:completion} combines it with the bulk law in \cite{HayaseCollinsKarakida2026} and transfers
the result to every fixed $s_k(A)$.  \Cref{sec:square} specializes the answer
to the square model.  Appendix~\ref{app:numerical-edge} gives a numerical
evaluation of its edge with explicit error bounds.

\section{Model and main result}
\label{sec:model}

For each $d$, let
\[
  X=X_d\in\R^{\ell\times d},
  \qquad XX^\top=I_\ell,
  \qquad \ell\leq d,
\]
be deterministic.  Let
\(
  W^Q,W^K\in\R^{d\times p}
\)
be independent matrices with independent $\mathcal N(0,1)$ entries, where
$p=d_{qk}$.  Put
\begin{equation*}
  Q=XW^Q,\qquad K=XW^K,\qquad
  S=\frac{QK^\top}{\sqrt p}\in\R^{\ell\times\ell}.
\end{equation*}
For a fixed $\beta>0$, define row-softmax attention by
\begin{equation*}
  A_{ij}=\frac{e^{\beta S_{ij}}}
  {\sum_{r=1}^{\ell}e^{\beta S_{ir}}}.
\end{equation*}
Consider the proportional regime
\begin{equation*}
  d,\ell,p\longrightarrow\infty,
  \qquad \frac\ell d\longrightarrow\gamma\in(0,1],
  \qquad \frac p d\longrightarrow\psi\in(0,\infty),
\end{equation*}
and write
\begin{equation*}
  \rho:=\lim\frac\ell p=\frac\gamma\psi\in(0,\infty).
\end{equation*}
Since $XX^\top=I_\ell$, both $Q$ and $K$ have independent standard Gaussian
entries.  Thus the law of the reduced model depends on $(\gamma,\psi)$ only
through $\rho$.  The three-parameter notation of \cite{HayaseCollinsKarakida2026} is retained for
$\nu_\infty$, while its edge is written as $E_+(\beta,\rho)$.

Let
\begin{equation*}
  u=u_\ell:=\ell^{-1/2}\one,
  \qquad P:=I_\ell-uu^\top.
\end{equation*}
For a square matrix $Y$, let
\[
  \nu_Y:=\frac1\ell\sum_{j=1}^{\ell}\delta_{s_j(Y)^2}
\]
denote its empirical squared singular-value law.  By
\cite[Theorem~3.2]{HayaseCollinsKarakida2026},
\begin{equation*}
  \nu_{\sqrt\ell(A-uu^\top)}
  \xrightarrow[\ell\to\infty]{\text{moments, a.s.}}
  \nu_\infty(\beta,\gamma,\psi),
\end{equation*}
where $\nu_\infty$ is deterministic and compactly supported.
In particular, the Fr\'echet--Shohat moment theorem
\cite{FrechetShohat1931} gives almost-sure weak convergence, because the
compactly supported limit $\nu_\infty$ is uniquely determined by its
moments.  No common compact support for the empirical measures is
required.  The pathwise tightness, moment identification, and uniqueness
argument is given in \cref{subsec:bulk-lower}; see also the moment method
in \cite[Section~1.3.1]{BaiSilverstein2010}.

The notation $X_\ell=o_{\Pp}(a_\ell)$ means
$X_\ell/a_\ell\to0$ in probability, while $X_\ell=O_{\Pp}(a_\ell)$ means that
$X_\ell/a_\ell$ is bounded in probability (tight).  For a real matrix $B$, set
$|B|:=(B^\top B)^{1/2}$, the positive semidefinite matrix absolute value,
not the entrywise absolute value.  We use
\[
  \lmaxabs{B}=s_1(B)
  =\sqrt{\lambda_{\max}(B^\top B)}
\]
to express its operator norm in terms of a largest eigenvalue.  If $B$ is
positive semidefinite, this is simply $\lambda_{\max}(B)$; for a symmetric
matrix with both signs of eigenvalues, $\lambda_{\max}(B)$ alone would not
give its operator norm.  For matrices, $o_{\Pp}$ and $O_{\Pp}$ refer to this
quantity unless another norm is specified.  Vector norms are Euclidean,
and $\norm{\cdot}_{\rm F}$ denotes the Frobenius norm.

\begin{theorem}[Bulk-edge sticking beyond the Perron mode]
\label{thm:main}
Under the model and proportional regime specified above, for every fixed
$\beta>0$ and every fixed integer $k\geq2$,
\begin{equation*}
  \ell s_k(A)^2\xrightarrow{\Pp}
  E_+(\beta,\rho)
  :=\max\supp\nu_\infty(\beta,\gamma,\psi).
\end{equation*}
More strongly, for every deterministic integer sequence
$r_\ell=o(\ell)$ with $2\leq r_\ell\leq\ell-1$,
\begin{equation*}
 \max_{2\leq k\leq r_\ell}
 \left|\ell s_k(A)^2-E_+(\beta,\rho)\right|
 \xrightarrow{\Pp}0.
\end{equation*}
Consequently there is no persistent non-Perron singular-value outlier in
this fixed-temperature Gaussian model.
\end{theorem}

\begin{remark}[Scale in terms of the embedding dimension]
Since $\ell/d\to\gamma$,
\[
  d s_k(A)^2\xrightarrow{\Pp}\frac1\gamma
  E_+(\beta,\gamma/\psi),\qquad k\geq2\text{ fixed}.
\]
The natural scale is $\ell$, not $d$.
\end{remark}

\section{The limiting edge}
\label{sec:edge-description}

Set
\begin{equation*}
  a:=e^{\beta^2}-1-\beta^2>0,
  \qquad b:=\beta^2.
\end{equation*}
Let $\MP_\rho$ denote the Marchenko--Pastur law associated with
$GG^\top/p$ when $G$ is $\ell\times p$ standard Gaussian and
$\ell/p\to\rho$.  Its continuous support is
\([
  (1-\sqrt\rho)^2,(1+\sqrt\rho)^2
]\), with an additional atom of mass $1-\rho^{-1}$ at zero when $\rho>1$;
see \cite{MarchenkoPastur1967} and \cite[Chapter~3]{BaiSilverstein2010}.
Define
\begin{equation*}
  H_{\beta,\rho}:=\operatorname{Law}(a+bX),
  \qquad X\sim\MP_\rho,
\end{equation*}
and
\begin{equation*}
  L_{\beta,\rho}:=a+b(1+\sqrt\rho)^2
  =\max\supp H_{\beta,\rho}.
\end{equation*}

The Gaussian equivalent in \cite[Eq.~(3.4)]{HayaseCollinsKarakida2026} is
\begin{equation}
  Y_{\rm lin}
  =\frac{\beta}{\sqrt{\ell p}}QK^\top
   +\frac{\sqrt a}{\sqrt\ell}W,
  \label{eq:gaussian-equivalent}
\end{equation}
where $W$ is an independent $\ell\times\ell$ standard Gaussian matrix.

The following proposition is a specialization of the classical
sample-covariance limit theorem
\cite[Theorem~1.1]{Silverstein1995} to this Gaussian equivalent, rather
than a new bulk-limit theorem.  It records the known law in the
population--sample form needed for the edge analysis below.  The same
affine-Marchenko--Pastur convolution appears in
\cite[Appendix~I, Eq.~(66)]{FanWang2020}: after scaling by $a+b$, their
parameters are $\gamma_0=\rho$, $\gamma_1=1$, and
$b_\sigma^2=b/(a+b)$.  Here the argument uses the Gaussian equivalent
directly, so it does not require applying their nonlinear-feature theorem
to the exponential activation.

\begin{proposition}[Sample-covariance representation of the known bulk law]
\label[proposition]{prop:bulk-law-identification}
Applying \cite[Theorem~1.1]{Silverstein1995} to the Gaussian equivalent
\eqref{eq:gaussian-equivalent}, with limiting population law
$H_{\beta,\rho}$ and sample aspect ratio one, gives
\begin{equation}
  \nu_\infty=H_{\beta,\rho}\boxtimes\MP_1.
  \label{eq:nu-free-convolution}
\end{equation}
In particular, the outer Marchenko--Pastur aspect ratio is one; the ratio
$\rho$ occurs inside the population law $H_{\beta,\rho}$.
\end{proposition}

\begin{proof}
We use the conditional distribution given the key matrix $K$ and write
$q_i\in\R^p$ and $w_i\in\R^\ell$ for the
independent Gaussian vectors generating the $i$th row of the two terms in
\eqref{eq:gaussian-equivalent}.  That row is
\[
 \frac1{\sqrt\ell}
 \left(\frac\beta{\sqrt p}q_i^\top K^\top+\sqrt a\,w_i^\top\right),
\]
so its conditional covariance is
\[
 \frac1\ell T_K,
 \qquad T_K:=aI_\ell+b\frac{KK^\top}{p}.
\]
The rows are conditionally independent.  Consequently, exactly in
conditional distribution,
\begin{equation*}
  Y_{\rm lin}\ \overset{d}{=}\
  \frac1{\sqrt\ell}G T_K^{1/2},
\end{equation*}
where $G\in\R^{\ell\times\ell}$ is standard Gaussian and independent of
$K$.  The Marchenko--Pastur theorem~\cite{MarchenkoPastur1967} gives
\(
 \ESD(T_K)\Rightarrow H_{\beta,\rho}
\)
almost surely.  To apply \cite[Theorem~1.1]{Silverstein1995}, take its
population matrix to be $T_K$, its independent-entry matrix to be
$G^\top$, and its dimensions to be $n=N=\ell$.  The matrix $T_K$ is
positive definite and independent of $G$, its empirical distribution
has the limit just obtained, and the entries of $G$ are independent,
centered, and of variance one.  Thus all the hypotheses of that theorem
hold.  Since $Y_{\rm lin}^\top Y_{\rm lin}$ has the conditional law of
$\ell^{-1}T_K^{1/2}G^\top G T_K^{1/2}$, the theorem gives the squared
singular-value law
$H_{\beta,\rho}\boxtimes\MP_1$, where $\boxtimes$ denotes free
multiplicative convolution.  The same Gaussian equivalent defines
$\nu_\infty$ in \cite[Eq.~(3.7)]{HayaseCollinsKarakida2026}, so uniqueness
of weak limits proves
\eqref{eq:nu-free-convolution}.
\end{proof}

For $0<q<L_{\beta,\rho}^{-1}$, introduce the positive-variable
inverse Stieltjes map
\begin{equation*}
  \Phi_{\beta,\rho}(q)
  :=\frac1q+\int\frac{t}{1-qt}\,dH_{\beta,\rho}(t).
\end{equation*}
Its derivatives are
\begin{align*}
  \Phi_{\beta,\rho}'(q)
  &=-\frac1{q^2}
    +\int\frac{t^2}{(1-qt)^2}\,dH_{\beta,\rho}(t), 
\\
  \Phi_{\beta,\rho}''(q)
  &=\frac2{q^3}
    +2\int\frac{t^3}{(1-qt)^3}\,dH_{\beta,\rho}(t)>0. 
\end{align*}
As $q\downarrow0$, the first derivative tends to $-\infty$; as
$q\uparrow L_{\beta,\rho}^{-1}$, its integral term diverges to $+\infty$
because $H_{\beta,\rho}$ has positive mass in every left neighborhood of
$L_{\beta,\rho}$ with the Marchenko--Pastur square-root edge density.
Hence there is a unique $q_*=q_*(\beta,\rho)$ satisfying
\begin{equation*}
  \frac1{q_*^2}
  =\int\frac{t^2}{(1-q_*t)^2}\,dH_{\beta,\rho}(t).
\end{equation*}
The support characterization in
\cite[Theorems~4.1--4.2]{SilversteinChoi1995}, with $q=-m$ in its inverse
Stieltjes transform, gives
\begin{equation}
  E_+(\beta,\rho)
  =\Phi_{\beta,\rho}(q_*),
  \qquad \rho=\gamma/\psi.
  \label{eq:E-Phi}
\end{equation}

For computation, let
\begin{equation*}
  M_\rho(z):=\int\frac1{1-zx}\,d\MP_\rho(x).
\end{equation*}
The explicit Marchenko--Pastur transform formula
\cite[Chapter~3]{BaiSilverstein2010}, on the physical branch $M_\rho(0)=1$, gives
\begin{equation*}
 M_\rho(z)=
 \frac{1-(1-\rho)z-
 \sqrt{\bigl(1-(1+\sqrt\rho)^2z\bigr)
              \bigl(1-(1-\sqrt\rho)^2z\bigr)}}
 {2\rho z}.
\end{equation*}
Since
\[
  1-q(a+bx)=(1-aq)
  \left(1-\frac{bq}{1-aq}x\right),
\]
the elementary identity
\(
  q^{-1}+t/(1-qt)=1/[q(1-qt)]
\)
gives
\begin{equation}
  \Phi_{\beta,\rho}(q)
  =\frac1{q(1-aq)}
  M_\rho\!\left(\frac{bq}{1-aq}\right).
  \label{eq:Phi-M}
\end{equation}

The next proposition removes the integral from the edge computation.

\begin{proposition}[Scalar formula for the proportional edge]
\label[proposition]{prop:scalar-edge}
Let $\kappa:=a/b$ and let $\eta_*$ be the unique root in
$(0,\rho^{-1/2})$ of
\begin{equation}
  2\rho^2\eta^3+3\rho\eta^2
  +(\kappa+1-\rho)\eta-1=0.
  \label{eq:eta-cubic}
\end{equation}
With
\begin{equation*}
  D_*:=1+(\kappa+1+\rho)\eta_*+\rho\eta_*^2,
\end{equation*}
one has
\begin{equation*}
  q_*=\frac{\eta_*}{bD_*},
  \qquad
  E_+(\beta,\rho)
  =\frac{bD_*^2}{\eta_*(1+\rho\eta_*)},
  \quad \rho=\gamma/\psi.
\end{equation*}
\end{proposition}

\begin{proof}
The moment transform satisfies the standard Marchenko--Pastur identity
\cite[Chapter~3]{BaiSilverstein2010}
\begin{equation*}
  M_\rho(z)=1+zM_\rho(z)
  \bigl(1-\rho+\rho M_\rho(z)\bigr).
\end{equation*}
Put $\eta=M_\rho(z)-1$.  Solving it for $z$ gives
\begin{equation}
  z=\frac{\eta}{(1+\eta)(1+\rho\eta)}.
  \label{eq:z-eta}
\end{equation}
The physical branch is exactly $0<\eta<\rho^{-1/2}$: on this interval
\begin{equation*}
  \frac{dz}{d\eta}
  =\frac{1-\rho\eta^2}
  {(1+\eta)^2(1+\rho\eta)^2}>0,
\end{equation*}
and the endpoint $\eta=\rho^{-1/2}$ maps to
$(1+\sqrt\rho)^{-2}$, the first singularity of $M_\rho$.

Equating \eqref{eq:z-eta} with $z=bq/(1-aq)$ yields
\begin{equation}
  q=q(\eta)=\frac{\eta}{bD(\eta)},
  \qquad
  D(\eta)=1+(\kappa+1+\rho)\eta+\rho\eta^2.
  \label{eq:q-eta}
\end{equation}
Substitution in \eqref{eq:Phi-M}, using $M_\rho(z)=1+\eta$, gives
\begin{equation}
  \Phi_{\beta,\rho}(q(\eta))
  =\frac{bD(\eta)^2}{\eta(1+\rho\eta)}.
  \label{eq:Phi-eta}
\end{equation}
Also
\begin{equation*}
  q'(\eta)=\frac{1-\rho\eta^2}{bD(\eta)^2}>0.
\end{equation*}
Thus $\Phi'(q)=0$ is equivalent to differentiating the right side of
\eqref{eq:Phi-eta} with respect to $\eta$.  After clearing its positive
denominator, the derivative is zero precisely when
\eqref{eq:eta-cubic} holds.  The cubic is negative at zero and positive at
$\rho^{-1/2}$; uniqueness in the physical interval also follows from the
strict convexity of $\Phi_{\beta,\rho}$ and the positivity of $q'(\eta)$.  The formulas for $q(\eta)$ and $\Phi_{\beta,\rho}(q(\eta))$ give the
claimed expressions for $q_*$ and $E_+$.
\end{proof}

\section{Exact softmax reduction}
\label{sec:softmax}

Define the centered exponential feature matrix $F\in\R^{\ell\times\ell}$
by
\begin{equation*}
  F_{ij}:=e^{\beta S_{ij}-\beta^2/2}-1,
\end{equation*}
and define
\begin{equation*}
  r_i:=\frac1\ell\sum_{j=1}^{\ell}F_{ij},
  \qquad
  R:=\diag(r_1,\ldots,r_\ell)
  =\diag\!\left(\frac1\ell F\one\right).
\end{equation*}
The normalization calculation is given explicitly.  By the definition of $F$,
\begin{equation*}
  e^{\beta S_{ij}}=e^{\beta^2/2}(1+F_{ij}).
\end{equation*}
Therefore the $i$th row normalizer is
\begin{align*}
  Z_i
  &:=\sum_{j=1}^{\ell}e^{\beta S_{ij}}
    =e^{\beta^2/2}\sum_{j=1}^{\ell}(1+F_{ij}) \\
  &=e^{\beta^2/2}
    \left(\ell+\sum_{j=1}^{\ell}F_{ij}\right)
    =e^{\beta^2/2}\ell(1+r_i). 
\end{align*}
Consequently
\begin{equation*}
  A_{ij}=\frac{1+F_{ij}}{\ell(1+r_i)}.
\end{equation*}
Since $uu^\top=\ell^{-1}\one\one^\top$, this is the exact identity
\begin{equation*}
  A=(I+R)^{-1}\left(uu^\top+\frac F\ell\right).
\end{equation*}
This identity is exact.  It also verifies $Au=u$:
\[
 \left(uu^\top+\frac F\ell\right)u
 =u+\frac{F\one}{\ell\sqrt\ell}
 =u+Ru=(I+R)u.
\]
Hence
\begin{equation*}
  A-uu^\top=A-Auu^\top=AP.
\end{equation*}
Right multiplication by $P$ gives
\begin{equation*}
  \sqrt\ell\,AP=(I+R)^{-1}\frac{FP}{\sqrt\ell}.
\end{equation*}
The double compression used below is
\begin{equation}
  C_\ell:=\sqrt\ell\,PAP
  =P(I+R)^{-1}\frac{FP}{\sqrt\ell}.
  \label{eq:C-exact}
\end{equation}

\section{Uniform concentration of the row normalizers}
\label{sec:normalizer}

The square-model estimate is proved in
\cite[Lemma~B.10 in the extended arXiv version]{HayaseCollinsKarakida2026}.
The following argument records
its rectangular extension, for which $\ell/p\to\rho$ is the relevant ratio.

\begin{lemma}[Rectangular normalizer estimate]
\label[lemma]{lem:normalizer}
For every fixed $\delta>0$,
\begin{equation*}
  \ell^{1/2-\delta}\lmaxabs{R}
  \xrightarrow{\Pp}0.
\end{equation*}
In fact the failure probability at any fixed positive threshold can be made smaller
than any prescribed negative power of $\ell$.
\end{lemma}

\begin{proof}
Write $q_i^\top$ for the $i$th row of $Q$ and set
\begin{equation*}
  v_i:=\frac{\norm{q_i}^2}{p}.
\end{equation*}
Conditionally on $q_i$, the random variables
\(
  S_{ij}=q_i^\top k_j/\sqrt p
\), $j=1,\ldots,\ell$, are independent $\mathcal N(0,v_i)$.  Hence
\begin{equation*}
  \mu_i:=\E_K[e^{\beta S_{ij}}\mid q_i]
  =e^{\beta^2v_i/2}.
\end{equation*}

We use Rosenthal's inequality in the following form: for independent,
mean-zero real random variables $X_1,\ldots,X_N$ and fixed $m>2$ with
finite $m$th moments,
\[
 \E\left|\sum_{j=1}^N X_j\right|^m
 \leq C_m\left\{
   \sum_{j=1}^N\E|X_j|^m
   +\left(\sum_{j=1}^N\E X_j^2\right)^{m/2}
 \right\};
\]
see \cite[Theorem~3]{Rosenthal1970}.  The case $m=2$ follows directly
from additivity of variance.  Apply this inequality conditionally on $Q$
with $N=\ell$ and $X_j=\xi_{ij}:=e^{\beta S_{ij}}-\mu_i$, for a fixed row
$i$.  The $\xi_{ij}$ are conditionally independent and centered.
For any fixed even integer $m\geq2$, the Gaussian moment-generating
formula and $|x-y|^m\leq2^{m-1}(|x|^m+|y|^m)$ give
\[
 \E_K[|\xi_{ij}|^m\mid Q]
 \leq2^{m-1}\left(e^{m^2\beta^2v_i/2}
                       +e^{m\beta^2v_i/2}\right)
 \leq C_{m,\beta}
\]
on $\max_i v_i\leq2$.  In particular, the conditional variances are
uniformly bounded there.  Dividing Rosenthal's bound by $\ell^m$ yields
\begin{align*}
 \E_K\!\left[
 \left|\frac1\ell\sum_{j=1}^{\ell}
 (e^{\beta S_{ij}}-\mu_i)\right|^m
 \middle|Q\right]
 &\leq C_{m,\beta}\ell^{-m}(\ell+\ell^{m/2})\\
 &\leq C_{m,\beta}\ell^{-m/2}.
\end{align*}
Markov's inequality and a union bound over the $\ell$ rows therefore yield,
for each $\varepsilon>0$,
\begin{equation*}
 \Pp_K\!\left(
 \max_i\left|\frac1\ell\sum_j
 (e^{\beta S_{ij}}-\mu_i)\right|
 >\varepsilon\ell^{-1/2+\delta}
 \middle|Q\right)
 \leq C_{m,\beta,\varepsilon}\ell^{1-m\delta}
\end{equation*}
on that event.  Since $m$ may be chosen arbitrarily large, the right side
has any desired polynomial decay.

It remains to replace the conditional means $\mu_i$ by $e^{\beta^2/2}$
and to control the event used above.  Laurent--Massart's chi-square
bounds~\cite[Lemma~1 and Eqs.~(4.3)--(4.4)]{LaurentMassart2000} state that,
for $Y\sim\chi_p^2$ and $x>0$,
\begin{align*}
 \Pp(Y-p\geq2\sqrt{px}+2x)&\leq e^{-x},\\
 \Pp(p-Y\geq2\sqrt{px})&\leq e^{-x}.
\end{align*}
Since $pv_i=\norm{q_i}^2\sim\chi_p^2$, division by $p$ gives
\[
 \Pp\left(
  |v_i-1|>2\sqrt{x/p}+2x/p
 \right)\leq2e^{-x}.
\]
Taking $x=(D+2)\log\ell$ and applying a union bound over the $\ell$ rows
gives, for every fixed $D>0$,
\[
 \Pp\!\left(
  \max_i|v_i-1|>
  2\sqrt{\frac{(D+2)\log\ell}{p}}
  +\frac{2(D+2)\log\ell}{p}
 \right)\leq2\ell^{-D-1}.
\]
In particular, $\max_i v_i\leq2$ except on an event with arbitrarily
small polynomial probability, for all sufficiently large $\ell$, and
\[
 \max_i|v_i-1|
 =O_{\Pp}\!\left(\sqrt{\frac{\log\ell}{p}}
                         +\frac{\log\ell}{p}\right)
 =o_{\Pp}(\ell^{-1/2+\delta}),
\]
because $p\asymp\ell$.  The mean-value theorem applied to
$v\mapsto e^{\beta^2v/2}$ shows that
\begin{equation*}
  \max_i|\mu_i-e^{\beta^2/2}|
  =o_{\Pp}(\ell^{-1/2+\delta}).
\end{equation*}
Finally,
\[
  r_i=e^{-\beta^2/2}
      \left(\frac1\ell\sum_j e^{\beta S_{ij}}\right)-1.
\]
Combining the conditional fluctuation bound with the replacement of
$\mu_i$ by $e^{\beta^2/2}$ proves the claim.
\end{proof}

On $\{\lmaxabs{R}\leq1/2\}$,
\begin{equation}
  \lmaxabs{(I+R)^{-1}-I}
  \leq2\lmaxabs{R}.
  \label{eq:inverse-bound}
\end{equation}
Indeed, $R$ is diagonal, and
$|(1+r_i)^{-1}-1|=|r_i|/|1+r_i|\leq2|r_i|$ for every $i$.
To remove the normalizer from \eqref{eq:C-exact}, the following bound will also be proved:
\begin{equation}
  \lmaxabs{FP/\sqrt\ell}=O_{\Pp}(1).
  \label{eq:F-tightness-needed}
\end{equation}

\section{Conditioning on the keys}
\label{sec:conditioning}

We compute expectations and covariances with respect to the conditional
distribution of the query matrix $Q$ given the key matrix $K$.  Since
$Q$ and $K$ are independent, the query rows $q_1,\ldots,q_\ell$ still have
their original independent standard Gaussian distributions under this
conditional law.  The rows of $F$ are consequently conditionally
independent and identically distributed, although the coordinates
within each row need not be independent.

\subsection{Conditional mean}

For fixed $j$, the score is a Gaussian linear functional of $q_i$:
\begin{equation*}
  S_{ij}\mid K\sim
  \mathcal N\!\left(0,\frac{\norm{k_j}^2}{p}\right).
\end{equation*}
The one-dimensional Gaussian moment-generating formula therefore gives
\begin{align*}
  \E_Q[e^{\beta S_{ij}}\mid K]
  &=\exp\!\left(\frac{\beta^2}{2}
                 \frac{\norm{k_j}^2}{p}\right), 
\\
  \E_Q[F_{ij}\mid K]
  &=e^{-\beta^2/2}
    \E_Q[e^{\beta S_{ij}}\mid K]-1 \\
  &=\exp\!\left[
      \frac{\beta^2}{2}
      \left(\frac{\norm{k_j}^2}{p}-1\right)
    \right]-1. 
\end{align*}
Define
\begin{equation*}
  D_j:=\exp\!\left[
      \frac{\beta^2}{2}
      \left(\frac{\norm{k_j}^2}{p}-1\right)
    \right],
  \quad m_j:=D_j-1,
  \quad D_K:=\diag(D_1,\ldots,D_\ell),
\end{equation*}
and, with $m=(m_1,\ldots,m_\ell)^\top$,
\begin{equation*}
  \widetilde F:=F-\one m^\top.
\end{equation*}
The rows of $\widetilde F$ are conditionally independent, centered, and
identically distributed.  Since $P\one=0$,
\begin{equation}
  PFP=P\widetilde FP.
  \label{eq:mean-killed}
\end{equation}
Thus the left Perron projection removes the conditional-mean spike exactly.

\subsection{Conditional covariance}

Let $\widetilde f_i^\top$ be a row of $\widetilde F$ and set
\begin{equation*}
  \Sigma_K:=\E_Q[\widetilde f_i\widetilde f_i^\top\mid K].
\end{equation*}
Put $X_{ij}:=e^{\beta S_{ij}-\beta^2/2}=1+F_{ij}$.  Subtracting constants
does not change covariance, so
\begin{equation*}
  \Sigma_K(j,k)
  =\operatorname{Cov}_Q(X_{ij},X_{ik}\mid K).
\end{equation*}
Conditionally on $K$, the pair $(S_{ij},S_{ik})$ is jointly Gaussian and
\begin{align*}
  \operatorname{Var}_Q(S_{ij}\mid K)&=\frac{\norm{k_j}^2}{p},
  &\operatorname{Var}_Q(S_{ik}\mid K)&=\frac{\norm{k_k}^2}{p}, 
\\
  \operatorname{Cov}_Q(S_{ij},S_{ik}\mid K)
  &=\frac{k_j^\top k_k}{p}. 
\end{align*}
Using the two-dimensional Gaussian moment-generating formula,
\begin{align*}
 \E_Q[X_{ij}X_{ik}\mid K]
 &=e^{-\beta^2}
   \E_Q[e^{\beta(S_{ij}+S_{ik})}\mid K] \\
 &=\exp\!\left\{-\beta^2+
   \frac{\beta^2}{2}\left(
   \frac{\norm{k_j}^2}{p}+\frac{\norm{k_k}^2}{p}
   +2\frac{k_j^\top k_k}{p}\right)\right\} \\
 &=D_jD_k\exp\!\left(\beta^2\frac{k_j^\top k_k}{p}\right). 
\end{align*}
Since $\E_Q[X_{ij}\mid K]=D_j$,
\begin{equation*}
  \Sigma_K(j,k)=D_jD_k
  \left[\exp\!\left(\beta^2\frac{k_j^\top k_k}{p}\right)-1\right].
\end{equation*}
Define the inner-product kernel matrix
\begin{equation*}
  M_K(j,k):=\exp\!\left(\beta^2\frac{k_j^\top k_k}{p}\right)-1.
\end{equation*}
Because $D_K$ is diagonal,
\[
  (D_KM_KD_K)_{jk}=D_jM_K(j,k)D_k=\Sigma_K(j,k).
\]
This proves the exact factorization
\begin{equation}
  \Sigma_K=D_KM_KD_K.
  \label{eq:Sigma-factorization}
\end{equation}

\section{Operator-norm approximation of the conditional covariance}
\label{sec:kernel}

The following specialization of the inner-product-kernel
approximation~\cite[Theorem~2.1]{ElKaroui2010} is used.

\begin{proposition}[Gaussian isotropic kernel approximation]
\label[proposition]{prop:El-Karoui}
Let $x_1,\ldots,x_n\in\R^p$ be independent standard Gaussian vectors, let
$n/p$ remain bounded above and below, and let $h$ be $C^3$ in a neighborhood
of zero and $C^1$ in a neighborhood of one.  If
\(
  \mathcal M_{jk}=h(x_j^\top x_k/p)
\), define
\begin{equation*}
 \mathcal L_h:=
 \left(h(0)+\frac{h''(0)}{2p}\right)\one\one^\top
 +h'(0)\frac{XX^\top}{p}
 +\bigl(h(1)-h(0)-h'(0)\bigr)I_n,
\end{equation*}
where $X$ has rows $x_j^\top$.  Then
\begin{equation*}
 \lmaxabs{\mathcal M-\mathcal L_h}
 \xrightarrow{\Pp}0,
\end{equation*}
that is, $\mathcal L_h$ approximates $\mathcal M$ in operator norm.
\end{proposition}

The substitution for the present model is
\[
  n=\ell,\qquad x_j=k_j,\qquad X=K,
  \qquad h(t)=e^{\beta^2t}-1.
\]
In particular,
\[
  h(0)=0,\quad h'(0)=\beta^2=b,\quad h''(0)=\beta^4,
  \quad h(1)-h(0)-h'(0)=a.
\]
All hypotheses of \cref{prop:El-Karoui} hold because the keys are Gaussian
and $\ell/p\to\rho\in(0,\infty)$.  Therefore
\begin{equation}
  \lmaxabs{M_K-\left(
  aI_\ell+b\frac{KK^\top}{p}
  +\frac{\beta^4}{2p}\one\one^\top
  \right)}
  \xrightarrow{\Pp}0.
  \label{eq:MK-approx}
\end{equation}

Let $U\in\R^{\ell\times(\ell-1)}$ have orthonormal columns spanning
$u^\perp$:
\begin{equation*}
  U^\top U=I_{\ell-1},\qquad UU^\top=P,
  \qquad U^\top\one=0.
\end{equation*}
The covariance after right compression is
\begin{equation*}
  \widehat\Sigma_K:=U^\top\Sigma_KU.
\end{equation*}
Apply the Laurent--Massart bounds used in \cref{lem:normalizer} to
$\norm{k_j}^2/p$, now taking a union bound over the keys.  Since
$D_j=\exp\{\beta^2(\norm{k_j}^2/p-1)/2\}$, the mean-value theorem gives
\begin{equation}
  \lmaxabs{D_K-I}
  =O_{\Pp}\!\left(\sqrt{\frac{\log\ell}{p}}
                      +\frac{\log\ell}{p}\right)=o_{\Pp}(1).
  \label{eq:D-close}
\end{equation}
For the Gaussian matrix $K\in\R^{\ell\times p}$,
\cite[Corollary~5.35]{Vershynin2012} gives the following quantitative
Wishart bound, for every $t\geq0$:
\[
 \Pp\!\left\{
   \lambda_{\max}\!\left(\frac{KK^\top}{p}\right)
   >\left(1+\sqrt{\frac\ell p}+\frac{t}{\sqrt p}\right)^2
 \right\}\leq2e^{-t^2/2}.
\]
This follows by squaring and dividing by $p$ the bound
$s_1(K)\leq\sqrt\ell+\sqrt p+t$ (or applying it to $K^\top$).
Taking $t=\sqrt p$ and using $\ell/p\to\rho$ proves
$\lambda_{\max}(KK^\top/p)=O_{\Pp}(1)$, and hence
$\lambda_{\max}(T_K)=O_{\Pp}(1)$ for $T_K=aI+bKK^\top/p$.
Writing $\Delta_K=D_K-I$, we have
\[
 D_KT_KD_K-T_K=\Delta_KT_K+T_K\Delta_K+\Delta_KT_K\Delta_K.
\]
The triangle inequality and submultiplicativity of the operator norm,
expressed in the notation of \cref{sec:model}, therefore give
\[
 \lmaxabs{D_KT_KD_K-T_K}
 \leq\bigl(2\lmaxabs{\Delta_K}+\lmaxabs{\Delta_K}^2\bigr)
        \lambda_{\max}(T_K)
 =o_{\Pp}(1).
\]
The rank-one term in
\eqref{eq:MK-approx} requires a separate estimate:
\begin{align}
 \frac1p\lambda_{\max}(U^\top D_K\one\one^\top D_KU)
 &=\frac1p\norm{U^\top D_K\one}^2 \notag\\
 &=\frac1p\norm{U^\top(D_K-I)\one}^2
 \leq\frac\ell p\lmaxabs{D_K-I}^2=o_{\Pp}(1).
 \label{eq:projected-rank-one}
\end{align}
Combining \eqref{eq:Sigma-factorization}, \eqref{eq:MK-approx}, and
\eqref{eq:projected-rank-one} gives
\begin{equation}
 \lmaxabs{\widehat\Sigma_K-\left[
 aI_{\ell-1}+b\frac{(U^\top K)(U^\top K)^\top}{p}
 \right]}
 \xrightarrow{\Pp}0.
 \label{eq:Sigma-hat-approx}
\end{equation}
Rotational invariance implies that $U^\top K$ is an
$(\ell-1)\times p$ matrix of independent standard Gaussian entries.

\section{Conditional local law and upper spectral confinement}
\label{sec:local-law}

The sample-covariance results in
\cite[Assumptions~1--2, Proposition~2.17(b), and
Corollary~2.7(a)]{FanMaPaquetteWang2026} are now applied.  Only the
outside-support confinement result is needed.

For clarity, the hypotheses to be checked are separated according to their
numbering in \cite[Section~2.1, p.~7]{FanMaPaquetteWang2026}.
\begin{itemize}[leftmargin=*,itemsep=3pt]
\item \emph{Assumption~1 (dimensions and population covariance):}
$n/N$ is bounded above and away from zero, $\Sigma$ is positive
semidefinite with uniformly bounded $\lambda_{\max}(\Sigma)$, and a
positive fraction of its eigenvalues is bounded away from zero.
\item \emph{Assumption~2 (random vectors):}
$\mathbf g_1,\ldots,\mathbf g_N\in\R^n$ are independent and centered
with common covariance $\Sigma$.  Each fixed-order Euclidean norm
moment grows at most polynomially in the dimension: for every integer
$r\geq1$ there is a constant $C_r$, independent of $n$ and the sample
index, such that $\E\norm{\mathbf g_i}_2^r\leq n^{C_r}$.
This is the meaning of \emph{polynomial norm moments}; the exponent may
depend on $r$, and the phrase does not mean that the activation is a
polynomial.  Assumption~2 also requires the quadratic-form concentration
bound below.
\end{itemize}
\begin{equation}
 \Pp\left(
 \left|\mathbf g_i^\top B\mathbf g_i-\Tr(\Sigma B)\right|
 \geq n^\varepsilon\norm B_{\rm F}
 \right)\leq C_{\varepsilon,D}n^{-D}
 \label{eq:quadratic-form-concentration}
\end{equation}
Here $B\in\R^{n\times n}$ is arbitrary and deterministic, and the bound
holds for every $\varepsilon,D>0$, uniformly in $i$ and $B$.
The scalar $\mathbf g_i^\top B\mathbf g_i$ is a quadratic form;
$\Tr(\Sigma B)$ is its expectation.  Thus
\eqref{eq:quadratic-form-concentration} bounds its fluctuation about its
expectation on the scale of the Frobenius norm
$\norm B_{\rm F}=(\sum_{j,k}|B_{jk}|^2)^{1/2}$, up to an arbitrarily
small power of the dimension.  The matrix is fixed before sampling the
vectors; the statement is not a simultaneous bound over matrices chosen
from those vectors.  Under these assumptions,
Corollary~2.7(a) confines every eigenvalue of
$N^{-1}\sum_i\mathbf g_i\mathbf g_i^\top$ to a fixed neighborhood of the support of its
finite-dimensional deformed Marchenko--Pastur law, with arbitrarily high
polynomial probability.  Its additional cumulant Assumption~3 is not
needed for this conclusion.

We first identify the conditional random-feature model, then check
Assumption~1, the norm moments in Assumption~2, and its quadratic-form
bound.  Finally we transfer these checks to the compressed vectors and
apply Corollary~2.7(a).  The notation correspondence is displayed below.
The bold upright vectors $\mathbf w$ and $\mathbf g_i$ follow the
typography of \cite{FanMaPaquetteWang2026}; the two rows below the second
rule are additional objects introduced in the present paper.

\begin{center}
\small
\begin{tabularx}{0.94\textwidth}{@{}p{0.40\textwidth}X@{}}
\toprule
Notation in \cite{FanMaPaquetteWang2026} & Conditional attention problem \\
\midrule
output dimension $n$ & $\ell$ before compression \\
latent dimension $d$ & $p$ \\
sample count $N$ & $\ell$ query rows \\
deterministic feature matrix $X$ & $K/\sqrt p\in\R^{\ell\times p}$ \\
latent vector $\mathbf w$ & $q_i\sim\mathcal N(0,I_p)$ \\
coordinate activation $\sigma_j(x)$
& $e^{\beta x-\beta^2/2}-D_j$ \\
random feature $\mathbf g_i$ & $\widetilde f_i$ \\
population covariance $\Sigma$ & $\Sigma_K=D_KM_KD_K$ \\
\midrule
compressed feature & $h_i=U^\top\widetilde f_i\in\R^{\ell-1}$ \\
compressed sample covariance &
$\ell^{-1}Z^\top Z$, $Z=\widetilde FU$ \\
\bottomrule
\end{tabularx}
\end{center}

\paragraph{Step 1: conditional random-feature representation.}
For each fixed realization of $K$,
\begin{equation*}
  \widetilde F_{ij}
  =\sigma_{j,K}\!\left((Kq_i/\sqrt p)_j\right),
  \qquad
  \sigma_{j,K}(x):=e^{\beta x-\beta^2/2}-D_j.
\end{equation*}
By \cref{sec:conditioning}, the vectors $\widetilde f_i$ are independent
under the conditional law of $Q$ given $K$, have conditional mean zero,
and have common covariance $\Sigma_K$.  This checks the independence,
centering, and covariance requirements in Assumption~2 before compression.

\paragraph{Step 2: dimensions and covariance bounds (Assumption 1).}
Before compression $n=N=\ell$; afterward $n=\ell-1$ and $N=\ell$.
Thus both sample aspect ratios stay bounded above and away from zero.
To check the population eigenvalue conditions, write the kernel
approximation as
\begin{equation}
 M_K=A_K^\circ+E_K,
 \qquad
 A_K^\circ:=aI_\ell+b\frac{KK^\top}{p}
       +\frac{\beta^4}{2p}\one\one^\top,
 \qquad \lmaxabs{E_K}=o_{\Pp}(1).
 \label{eq:MK-error-form}
\end{equation}
On the event
\(
 \lmaxabs{E_K}\leq a/4
\)
and
\(
 \lmaxabs{D_K-I}\leq1/2
\), Weyl's inequality~\cite[Section~4.3]{HornJohnson2013} and $A_K^\circ\succeq aI$ give
\begin{equation}
 \lambda_{\min}(M_K)\geq\frac{3a}{4},
 \qquad
 \lambda_{\min}(\Sigma_K)\geq
 \lambda_{\min}(D_K)^2\lambda_{\min}(M_K)
 \geq\frac{3a}{16}.
 \label{eq:population-lower}
\end{equation}
If additionally $\sqrt{\lambda_{\max}(KK^\top/p)}\leq C_K$ and
$c_\rho\leq\ell/p\leq C_\rho$,
then \eqref{eq:MK-error-form} also gives a deterministic upper bound on
$\lambda_{\max}(M_K)$, hence on $\lambda_{\max}(\Sigma_K)$.  Consequently, there are
constants $0<c_{\beta,\rho}<C_{\beta,\rho}<\infty$ such that, on events of probability
tending to one,
\begin{equation}
  c_{\beta,\rho} I\preceq\Sigma_K\preceq C_{\beta,\rho} I.
  \label{eq:population-bounds}
\end{equation}
Compression gives the same bounds for
$\widehat\Sigma_K=U^\top\Sigma_KU$.  In particular, all its population
eigenvalues, rather than merely a positive fraction, are bounded away
from zero.  This verifies Assumption~1 on the events just described.

\paragraph{Step 3: polynomial norm moments (Assumption 2).}
We give an explicit bound, stronger than the required polynomial growth.
If $\sqrt{\lambda_{\max}(KK^\top/p)}\leq C_K$, then
$\norm{k_j}^2/p\leq C_K^2$ for every $j$.  Write
$V_{ij}=e^{\beta q_i^\top k_j/\sqrt p-\beta^2/2}$, so that
$\widetilde F_{ij}=V_{ij}-D_j$ and $D_j=\E_Q[V_{ij}\mid K]$.
For every integer $r\geq1$, the Gaussian exponential-moment formula gives
\[
 \E_Q[V_{ij}^{2r}\mid K]
 =\exp\!\left(-r\beta^2+
       2r^2\beta^2\frac{\norm{k_j}^2}{p}\right)
 \leq\exp\!\left(-r\beta^2+2r^2\beta^2C_K^2\right).
\]
Jensen's inequality and $|x-y|^{2r}\leq
2^{2r-1}(|x|^{2r}+|y|^{2r})$ consequently imply
$\E_Q[|\widetilde F_{ij}|^{2r}\mid K]\leq C_{r,\beta,C_K}$,
uniformly in $i,j,\ell$ and such key matrices.  Hence
\[
 \E_Q[\norm{\widetilde f_i}^{2r}\mid K]
 \leq\ell^{r-1}\sum_{j=1}^{\ell}
       \E_Q[|\widetilde F_{ij}|^{2r}\mid K]
 \leq C_{r,\beta,C_K}\ell^r.
\]
Bounds for every other fixed positive moment follow from these even
moments by H\"older's inequality.  The resulting bound has at most
polynomial growth in $\ell$, as required in Assumption~2.

\paragraph{Step 4: quadratic-form concentration (Assumption 2).}
We now use \cite[Proposition~2.17(b)]{FanMaPaquetteWang2026} to obtain
\eqref{eq:quadratic-form-concentration} for $\widetilde f_i$.
The feature matrix is $X=K/\sqrt p$, whose operator norm is
$\sqrt{\lambda_{\max}(KK^\top/p)}$.  The Gaussian matrix bound
\cite[Corollary~5.35]{Vershynin2012}, applied in \cref{sec:kernel},
bounds this quantity with probability tending to one, and
$p/\ell\to1/\rho$ verifies the dimension requirement $d\leq Cn$.
The latent vector is standard Gaussian.  The activation has the
ordinary power series
\begin{equation}
 \sigma_{j,K}(x)
 =(e^{-\beta^2/2}-D_j)
 +e^{-\beta^2/2}\sum_{r\geq1}\frac{\beta^r}{r!}x^r.
 \label{eq:activation-series}
\end{equation}
For every fixed exponent $\alpha\in(1/2,1)$,
\[
  \sup_{r\geq1}\frac{|\beta|^r}{(r!)^{1-\alpha}}<\infty.
\]
Thus its nonconstant coefficients obey the $(r!)^{-\alpha}$ envelope
in that proposition; \eqref{eq:D-close} bounds the constant coefficients.
The exponent $\alpha$ here plays the role of the coefficient-decay
parameter called $\beta$ in that proposition, not the role of our
inverse temperature.  The criterion permits coordinate-dependent
activations, and Step~2 has already checked its covariance hypothesis.
It therefore supplies the desired tail bound for
$\widetilde f_i^\top B\widetilde f_i-\Tr(\Sigma_KB)$, conditionally
on each key matrix satisfying these bounds.

\paragraph{Step 5: transfer of Assumptions 1--2 to the compressed vectors.}
The vectors $h_i=U^\top\widetilde f_i$ remain conditionally independent
and centered with common covariance $\widehat\Sigma_K$; Step~2 already
checked Assumption~1 for this covariance.  To transfer the quadratic-form
bound in Assumption~2, take any deterministic
$B\in\R^{(\ell-1)\times(\ell-1)}$.  Then
\begin{align*}
 h_i^\top Bh_i-\Tr(\widehat\Sigma_KB)
 &=\widetilde f_i^\top(UBU^\top)\widetilde f_i
   -\Tr(\Sigma_KUBU^\top), 
\\
 \norm{UBU^\top}_{\rm F}&=\norm B_{\rm F}. 
\end{align*}
The first identity expresses the centered quadratic form of $h_i$ as
one of $\widetilde f_i$, with test matrix $UBU^\top$.  The second
identity shows that its Frobenius-norm scale is unchanged.  Applying
the bound from Step~4 to that test matrix therefore gives
\eqref{eq:quadratic-form-concentration} in dimension $n=\ell-1$.
Explicitly, use the uncompressed estimate with exponent
$\varepsilon/2$ and note that
$\ell^{\varepsilon/2}\leq(\ell-1)^\varepsilon$ for all sufficiently
large $\ell$; the failure bound changes only by a constant factor.
The polynomial norm-moment condition transfers because $U^\top$ is
a contraction: $\norm{h_i}\leq\norm{\widetilde f_i}$.  In particular,
for each fixed integer $r\geq1$,
\[
 \E_Q[\norm{h_i}^{2r}\mid K]
 \leq C_{r,\beta,C_K}\ell^r
 \leq C'_{r,\beta,C_K}(\ell-1)^r\qquad(\ell\geq2).
\]
After increasing the exponent to absorb the constant, this is, for
all sufficiently large $\ell$, the bound by a power of $n=\ell-1$
required in Assumption~2.

\paragraph{Step 6: the finite reference law and spectral confinement.}
For a fixed realization of $K$, set
\[
 H_{\ell,K}:=\ESD(\widehat\Sigma_K),\qquad
 c_\ell:=\frac{\ell-1}{\ell},\qquad
 \mu_{\ell,K}:=H_{\ell,K}\boxtimes\MP_{c_\ell}.
\]
This is the \emph{finite-dimensional deformed Marchenko--Pastur law}:
the reference probability measure constructed from the current
population eigenvalues and current dimension-to-sample ratio, rather
than their limits.  Given $K$, it is deterministic; it is not the
empirical eigenvalue measure of the random matrix $\ell^{-1}Z^\top Z$.
It corresponds to $\mu_0$ in
\cite[Section~2.2, Eqs.~(2.4)--(2.6)]{FanMaPaquetteWang2026}, with
population covariance $\widehat\Sigma_K$, output dimension $\ell-1$,
and sample count $\ell$.  Their companion law is
$\widetilde\mu_{\ell,K}=c_\ell\mu_{\ell,K}+(1-c_\ell)\delta_0$;
the two measures have the same positive support.  Define
\[
 e_\ell(K):=\max\supp\mu_{\ell,K}
           =\max\supp\widetilde\mu_{\ell,K}.
\]
Thus $e_\ell(K)$ is the upper endpoint of the \emph{support} of the
finite deformed Marchenko--Pastur law, not a sample eigenvalue.
The conditional application can be stated uniformly as follows.

\begin{lemma}[Quenched upper confinement]
\label[lemma]{lem:quenched}
There are deterministic Borel sets of key matrices
$\mathcal G_\ell\subset\R^{\ell\times p}$ with
$\Pp(K\in\mathcal G_\ell)\to1$ such that, for every fixed $\varepsilon>0$,
\begin{equation}
 \sup_{K\in\mathcal G_\ell}
 \Pp_Q\!\left(
 \lmaxabs{\widetilde FU/\sqrt\ell}^2>e_\ell(K)+\varepsilon
 \ \middle|\ K\right)\longrightarrow0.
 \label{eq:quenched}
\end{equation}
The event $\{K\in\mathcal G_\ell\}$ is measurable with respect to
$\sigma(K)$: whether it occurs is determined solely by the key matrix,
without using $Q$.  The supremum is over fixed key realizations in
$\mathcal G_\ell$, and the conditional probability is over the remaining
Gaussian randomness in $Q$.
\end{lemma}

\begin{proof}
Fix deterministic constants
$C_K,c_\rho,C_\rho,c_{\beta,\rho},C_{\beta,\rho}$ for which the
following conditions hold with probability tending to one, and define
the set of admissible key realizations by
\begin{align*}
 \mathcal G_\ell:=\bigl\{K:\;&
 \sqrt{\lambda_{\max}(KK^\top/p)}\leq C_K,
 \quad c_\rho\leq\ell/p\leq C_\rho,
 \quad \lmaxabs{D_K-I}\leq\tfrac12,
 \quad \lmaxabs{E_K}\leq\tfrac a4,\\[-2mm]
 &c_{\beta,\rho} I\preceq\Sigma_K\preceq C_{\beta,\rho} I,
 \quad
 c_{\beta,\rho} I\preceq\widehat\Sigma_K\preceq C_{\beta,\rho} I
 \bigr\}. 
\end{align*}
The Gaussian matrix bound~\cite[Corollary~5.35]{Vershynin2012}, \eqref{eq:D-close},
\eqref{eq:MK-error-form}, and \eqref{eq:population-lower} show that
$\Pp(K\in\mathcal G_\ell)\to1$.  On every such deterministic key matrix,
the operator norm $\sqrt{\lambda_{\max}(KK^\top/p)}$ of the normalized
key matrix $K/\sqrt p$, the coefficient envelope in
\eqref{eq:activation-series}, the aspect ratios, the covariance bounds,
and the norm-moment bounds of Step~3 are controlled by constants
independent of $\ell$ and of the chosen key realization.  Steps~4--5
then give Assumption~2 for the compressed vectors, while Step~2 gives
Assumption~1.  Every condition defining $\mathcal G_\ell$ is a Borel
condition on $K$, which also proves the stated measurability.

The supremum in \eqref{eq:quenched} is controlled by the following contradiction argument.  If that
supremum did not tend to zero, then for some $c>0$ there would be a
subsequence $\ell_m$ and deterministic matrices
$K_{\ell_m}\in\mathcal G_{\ell_m}$ for which the displayed conditional
failure probability is at least $c$.  These matrices, their coordinatewise
activations, and their covariances form a deterministic triangular array
with fixed bounds.  Proposition~2.17(b) of
\cite{FanMaPaquetteWang2026} applies to the uncompressed features along
this array.  Step~5 transfers Assumptions~1--2 to the compressed vectors,
so Corollary~2.7(a) of the same reference applies to those vectors.
It forces the same failure probabilities to tend to zero, a contradiction.
This proves the uniform conditional assertion.
\end{proof}

\paragraph{Step 7: inserting the left projection after confinement.}
Set
\begin{equation*}
  Z:=\widetilde FU\in\R^{\ell\times(\ell-1)}.
\end{equation*}
The local law applies to $\ell^{-1}Z^\top Z$.  The left projection $P$ is
inserted only afterward, because left multiplication mixes the conditionally
independent samples.  It can only decrease the Gram matrix:
\begin{equation*}
  \frac1\ell Z^\top PZ
  =\frac1\ell Z^\top Z
  -\left(\frac{Z^\top u}{\sqrt\ell}\right)
   \left(\frac{Z^\top u}{\sqrt\ell}\right)^\top
  \preceq\frac1\ell Z^\top Z.
\end{equation*}
Moreover,
\begin{equation*}
  \lmaxabs{P\widetilde FP/\sqrt\ell}^2
  =\lambda_{\max}\!\left(\frac1\ell Z^\top PZ\right).
\end{equation*}
Thus, for each fixed $\varepsilon>0$,
\begin{equation}
  \Pp\!\left(
  \lmaxabs{P\widetilde FP/\sqrt\ell}^2
  >e_\ell(K)+\varepsilon\right)\longrightarrow0.
  \label{eq:conditional-upper-before-edge}
\end{equation}

\section{Convergence of the finite conditional edge}
\label{sec:finite-edge}

This section identifies the limit of the finite conditional edge associated
with the reference law in \cite{FanMaPaquetteWang2026}.

\begin{lemma}[Conditional edge convergence]
\label[lemma]{lem:finite-edge}
With $e_\ell(K)$ defined above,
\begin{equation*}
  e_\ell(K)\xrightarrow{\Pp}
  E_+(\beta,\rho).
\end{equation*}
\end{lemma}

\begin{proof}
Let
$\lambda_{1,\ell}\geq\cdots\geq\lambda_{\ell-1,\ell}\geq0$
be the eigenvalues of $\widehat\Sigma_K$ and define
\begin{equation*}
  \Phi_\ell(q):=\frac1q+\frac1\ell
  \sum_{j=1}^{\ell-1}
  \frac{\lambda_{j,\ell}}{1-q\lambda_{j,\ell}},
  \qquad 0<q<\lambda_{1,\ell}^{-1}.
\end{equation*}
The coefficient is $1/\ell$, rather than $1/(\ell-1)$, because $\ell$ is
the conditional sample count.  In the notation of \cite[Eq.~(2.7)]{FanMaPaquetteWang2026},
$\Phi_\ell(q)=z_{0,\ell}(-q)$, where
\[
 z_{0,\ell}(m):=-\frac1m+\frac1\ell
 \sum_{j=1}^{\ell-1}\frac{\lambda_{j,\ell}}{1+m\lambda_{j,\ell}}.
\]
  On the high-probability event
\eqref{eq:population-bounds}, the covariance is positive definite.  Direct
differentiation yields
\begin{align*}
 \Phi_\ell'(q)
 &=-\frac1{q^2}+\frac1\ell\sum_{j=1}^{\ell-1}
   \frac{\lambda_{j,\ell}^2}{(1-q\lambda_{j,\ell})^2}, 
\\
 \Phi_\ell''(q)
 &=\frac2{q^3}+\frac2\ell\sum_{j=1}^{\ell-1}
   \frac{\lambda_{j,\ell}^3}{(1-q\lambda_{j,\ell})^3}>0. 
\end{align*}
Therefore there is a unique $q_\ell$ with
$\Phi_\ell'(q_\ell)=0$, and the support characterization
\cite[Theorems~4.1--4.2]{SilversteinChoi1995}, applied to the current
population eigenvalue measure and aspect ratio $(\ell-1)/\ell$, gives
\begin{equation*}
  e_\ell(K)=\Phi_\ell(q_\ell).
\end{equation*}

By \eqref{eq:Sigma-hat-approx}, the Marchenko--Pastur theorem, and the
Wishart upper-edge theorem~\cite{MarchenkoPastur1967,YinBaiKrishnaiah1988},
\begin{align*}
  \frac1{\ell-1}\sum_{j=1}^{\ell-1}
  \delta_{\lambda_{j,\ell}}
  &\Longrightarrow H_{\beta,\rho}, 
\\
  \lambda_{1,\ell}&\xrightarrow{\Pp}L_{\beta,\rho}. 
\end{align*}
Both convergences are in probability.  It follows that on every compact
$I\Subset(0,L_{\beta,\rho}^{-1})$,
\begin{equation}
  \Phi_\ell\longrightarrow\Phi_{\beta,\rho},
  \qquad
  \Phi_\ell'\longrightarrow\Phi_{\beta,\rho}'
  \quad\text{uniformly on $I$, in probability}.
  \label{eq:local-uniform-Phi}
\end{equation}
Indeed, the normalization difference contributes only the factor
$(\ell-1)/\ell\to1$.  Moreover, on a high-probability event all population
eigenvalues lie in one fixed compact interval, and the rational functions defining
$\Phi_\ell$ and $\Phi_\ell'$ are uniformly bounded
and equicontinuous on $I$.

Choose $q_-<q_*<q_+$ with
$[q_-,q_+]\Subset(0,L_{\beta,\rho}^{-1})$ and
\[
  \Phi_{\beta,\rho}'(q_-)<0<
  \Phi_{\beta,\rho}'(q_+).
\]
By \eqref{eq:local-uniform-Phi}, with probability tending to one the same
signs hold for $\Phi_\ell'$.  The convergence of the largest population eigenvalue also gives
$q_+<\lambda_{1,\ell}^{-1}$ with probability tending to one, so the entire
bracket lies in the domain of $\Phi_\ell$.  The strict increase of $\Phi_\ell'$ then forces
$q_-<q_\ell<q_+$.  Shrinking the bracket proves $q_\ell\to q_*$ in
probability, and another use of \eqref{eq:local-uniform-Phi} yields
\[
  e_\ell(K)=\Phi_\ell(q_\ell)
  \xrightarrow{\Pp}\Phi_{\beta,\rho}(q_*)
  =E_+(\beta,\rho).
\]
\end{proof}

Combining \eqref{eq:conditional-upper-before-edge} and
\cref{lem:finite-edge} gives
\begin{equation}
  \forall\varepsilon>0,\qquad
  \Pp\!\left(
  \lmaxabs{P\widetilde FP/\sqrt\ell}^2>E_++\varepsilon
  \right)\longrightarrow0.
  \label{eq:nonlinear-upper}
\end{equation}

\section{Completion of the proof}
\label{sec:completion}

\subsection{Removing the row normalizer}

First, \eqref{eq:F-tightness-needed} is verified.  From
\cref{lem:quenched,lem:finite-edge},
\begin{equation*}
  \lmaxabs{\widetilde FU/\sqrt\ell}=O_{\Pp}(1).
\end{equation*}
It remains to control $m$.  If $Y=\norm{k_j}^2\sim\chi_p^2$, then
\[
  D_j=e^{-\beta^2/2}e^{\beta^2Y/(2p)}.
\]
For completeness, if $g_1,\ldots,g_p$ are independent $\mathcal N(0,1)$
variables, then for $t<1/2$ a Gaussian integral gives
\[
 \E e^{t g_1^2}
 =\frac1{\sqrt{2\pi}}\int_{\R}e^{-(1-2t)x^2/2}\,dx
 =(1-2t)^{-1/2}.
\]
Independence and $Y\stackrel{d}=\sum_{r=1}^p g_r^2$ imply
$\E e^{tY}=(1-2t)^{-p/2}$.  Substituting $t=\beta^2/(2p)$ and
$t=\beta^2/p$, respectively, gives, for $p>2\beta^2$,
\begin{align*}
  \E[D_j]
  &=e^{-\beta^2/2}
    \left(1-\frac{\beta^2}{p}\right)^{-p/2}, 
\\
  \E[D_j^2]
  &=e^{-\beta^2}
    \left(1-\frac{2\beta^2}{p}\right)^{-p/2}. 
\end{align*}
Using $-\log(1-x)=x+x^2/2+O(x^3)$,
\begin{align*}
  \E[D_j]&=1+\frac{\beta^4}{4p}+O(p^{-2}), 
\\
  \E[D_j^2]&=1+\frac{\beta^4}{p}+O(p^{-2}). 
\end{align*}
Thus
\begin{equation*}
  \E\norm m^2
  =\ell\E(D_j-1)^2
  =\frac\ell p\frac{\beta^4}{2}+o(1)
  \longrightarrow\frac{\rho\beta^4}{2},
\end{equation*}
and therefore $\norm m=O_{\Pp}(1)$ by Markov's inequality.  Since
$FP=\one m^\top P+\widetilde FP$ and $P=UU^\top$,
\begin{align}
  \lmaxabs{FP/\sqrt\ell}
  &\leq\lmaxabs{\one m^\top P/\sqrt\ell}
        +\lmaxabs{\widetilde FUU^\top/\sqrt\ell} \notag\\
  &=\norm{Pm}+\lmaxabs{\widetilde FU/\sqrt\ell}
  =O_{\Pp}(1).
  \label{eq:FP-tight}
\end{align}

Equations \eqref{eq:C-exact}, \eqref{eq:inverse-bound}, and
\eqref{eq:FP-tight} now imply
\begin{align*}
 \lmaxabs{C_\ell-PFP/\sqrt\ell}
 &\leq2\lmaxabs{R}\lmaxabs{FP/\sqrt\ell}=o_{\Pp}(1).
\end{align*}
Using \eqref{eq:mean-killed},
\begin{equation*}
  C_\ell=P\widetilde FP/\sqrt\ell+o_{\Pp}(1)
  \quad\text{in operator norm}.
\end{equation*}
The right side before the error is operator-norm tight by
\eqref{eq:nonlinear-upper}.  Hence
\[
 \left|\lambda_{\max}(C_\ell^\top C_\ell)
       -\lmaxabs{P\widetilde FP/\sqrt\ell}^2\right|
 =o_{\Pp}(1).
\]
Thus \eqref{eq:nonlinear-upper}, used with $\varepsilon/2$, gives
\begin{equation}
  \forall\varepsilon>0,\qquad
  \Pp\bigl(\lambda_{\max}(C_\ell^\top C_\ell)>E_++\varepsilon\bigr)
  \longrightarrow0.
  \label{eq:C-upper}
\end{equation}

\subsection{The bulk supplies the matching lower bounds}
\label{subsec:bulk-lower}

Set
\begin{equation*}
  B_\ell:=\sqrt\ell(A-uu^\top)=\sqrt\ell AP.
\end{equation*}
By \cite[Theorem~3.2]{HayaseCollinsKarakida2026},
\begin{equation*}
  \nu_{B_\ell}\xrightarrow{\text{moments, a.s.}}\nu_\infty.
\end{equation*}

\paragraph{From moment convergence to weak convergence.}
We spell out the application of the Fr\'echet--Shohat moment theorem
\cite{FrechetShohat1931}.  For each integer $r\geq1$, the convergence of
the $r$th moment holds on an event of probability one.  Taking the
countable intersection of these events gives a single event on which
all moments converge simultaneously.  Fix an outcome in this event and
regard the empirical measures as a deterministic sequence.  Since each
moment sequence converges to a finite number,
\[
 c_r:=\sup_\ell\int_0^\infty x^r\,d\nu_{B_\ell}(x)<\infty,
 \qquad r\geq1.
\]
These constants may depend on the fixed outcome.  In particular,
\[
 \nu_{B_\ell}((R,\infty))\leq\frac{c_1}{R},
 \qquad R>0,
\]
so the sequence is tight.  By Prokhorov's theorem
\cite[Theorem~5.1]{Billingsley1999}, every subsequence has a further weakly
convergent subsequence; write $\mu$ for any such limit.  For each $r\geq1$,
the pointwise inequality $x^r\leq x^{r+1}/R$ for $x>R$ gives
\[
 \sup_\ell\int_{(R,\infty)}x^r\,d\nu_{B_\ell}(x)
 \leq\frac{c_{r+1}}{R}\longrightarrow0
 \qquad(R\to\infty).
\]
The same bound for $\mu$ follows from
$\int x^{r+1}\,d\mu\leq c_{r+1}$.  To justify this use of the
Portmanteau theorem~\cite[Theorem~2.1]{Billingsley1999} for an unbounded
function, apply its bounded-continuous-function characterization to
$x\mapsto\min\{x^{r+1},M\}$ on $[0,\infty)$.  Along the weakly
convergent subsequence, the integrals of this function converge and are
bounded by $c_{r+1}$.  Letting $M\to\infty$ and using the monotone
convergence theorem gives the claimed bound.
Approximating $x^r$ by the bounded continuous functions
$\min\{x^r,R^r\}$ on $[0,\infty)$, and then sending $R$ to infinity,
therefore shows that
\[
 \int x^r\,d\mu(x)=\int x^r\,d\nu_\infty(x),
 \qquad r\geq0.
\]
Finally, choose $L_0>0$ with
$\supp\nu_\infty\subset[0,L_0]$.  For every $R>L_0$ and integer $r\geq1$,
the matching moments imply
\[
 \mu((R,\infty))
 \leq R^{-r}\int x^r\,d\mu(x)
 \leq\left(\frac{L_0}{R}\right)^r.
\]
Letting $r\to\infty$ shows that $\mu$ is also supported on $[0,L_0]$.
Polynomial approximation on this compact interval then gives
$\mu=\nu_\infty$.  Thus every subsequential weak limit is
$\nu_\infty$, and tightness implies convergence of the entire sequence.
Since the argument holds on the chosen probability-one event,
\[
 \nu_{B_\ell}\Longrightarrow\nu_\infty
 \qquad\text{almost surely}.
\]
Only the limiting measure was assumed compactly supported, not the
supports of the empirical measures uniformly in $\ell$.  More generally,
the same argument applies whenever the limiting measure is uniquely
determined by its moments.

\paragraph{Transfer to the double compression.}
Since $C_\ell=PB_\ell$,
\begin{equation*}
  B_\ell^\top B_\ell-C_\ell^\top C_\ell
  =B_\ell^\top uu^\top B_\ell
\end{equation*}
is positive semidefinite of rank at most one.  The rank inequality for
empirical spectral distributions~\cite[Theorem~A.43]{BaiSilverstein2010}
gives
\[
 \sup_{t\in\R}\left|
 \nu_{B_\ell}(({-\infty},t])-
 \nu_{C_\ell}(({-\infty},t])\right|\leq\frac1\ell.
\]
Their distribution functions therefore have the same limit at every
continuity point of the distribution function of $\nu_\infty$.  Hence
\begin{equation}
  \nu_{C_\ell}\Longrightarrow\nu_\infty
  \qquad\text{almost surely}.
  \label{eq:C-global}
\end{equation}

\paragraph{Lower bounds at the edge.}
Let
\(
  \mu_{1,\ell}\geq\cdots\geq\mu_{\ell-1,\ell}\geq0
\)
be the eigenvalues of the restriction of
$C_\ell C_\ell^\top$ to $u^\perp$.  The upper bound
\eqref{eq:C-upper} applies simultaneously to every $\mu_{j,\ell}$.  For the
lower bound, fix $\varepsilon>0$.  Since $E_+=\max\supp\nu_\infty$,
\begin{equation*}
  \nu_\infty((E_+-\varepsilon,\infty))>0.
\end{equation*}
Weak convergence \eqref{eq:C-global} and the Portmanteau theorem
\cite[Theorem~2.1]{Billingsley1999} imply that,
almost surely,
\begin{equation*}
 \liminf_{\ell\to\infty}\frac1\ell
 \#\{j:\mu_{j,\ell}>E_+-\varepsilon\}
 \geq\nu_\infty((E_+-\varepsilon,\infty))>0.
\end{equation*}
Consequently, for every deterministic $r_\ell=o(\ell)$,
\begin{equation*}
 \Pp\bigl(\mu_{r_\ell,\ell}>E_+-\varepsilon\bigr)\longrightarrow1.
\end{equation*}
Together with \eqref{eq:C-upper}, this proves
\begin{equation*}
 \max_{1\leq j\leq r_\ell}
 \abs{\mu_{j,\ell}-E_+}\xrightarrow{\Pp}0.
\end{equation*}
In particular, for every fixed $j\geq1$,
\begin{equation*}
  \mu_{j,\ell}\xrightarrow{\Pp}E_+.
\end{equation*}

\subsection{Orthogonal compression of the Gram matrix}

Recall that $u=\ell^{-1/2}\one$ and that the columns of
$U\in\R^{\ell\times(\ell-1)}$ form an orthonormal basis of $u^\perp$.
Define the square matrix
\[
 O_\ell:=[\,u\ \ U\,]\in\R^{\ell\times\ell}.
\]
The identities $u^\top u=1$, $U^\top u=0$, $U^\top U=I_{\ell-1}$,
and $UU^\top=P=I_\ell-uu^\top$ give
\[
 O_\ell^\top O_\ell=O_\ell O_\ell^\top=I_\ell.
\]
Thus $O_\ell$ is orthogonal (and hence unitary when viewed as a complex
matrix).  The real symmetric positive semidefinite matrix
\[
 G_\ell:=O_\ell^\top(\ell AA^\top)O_\ell
\]
is orthogonally similar to $\ell AA^\top$, so its eigenvalues, in
decreasing order, are $\ell s_1(A)^2,\ldots,\ell s_\ell(A)^2$.
Deleting the first row and first column of $G_\ell$ gives the principal
submatrix $\ell U^\top AA^\top U$.

We identify this principal submatrix with the compressed Gram matrix.
Since $C_\ell=\sqrt\ell\,PAP$ and $PU=U$, set
\[
 \widehat C_\ell:=U^\top C_\ell U=\sqrt\ell\,U^\top AU.
\]
The row-stochastic identity $Au=u$ implies $U^\top Au=0$.
Inserting $I_\ell=P+uu^\top$ between $A$ and $A^\top$ therefore gives
\begin{align*}
 \ell U^\top AA^\top U
 &=\ell U^\top A(P+uu^\top)A^\top U\\
 &=\ell U^\top APA^\top U\\
 &=\widehat C_\ell\widehat C_\ell^\top.
\end{align*}
Moreover, $C_\ell P=C_\ell$ gives
\[
 \widehat C_\ell\widehat C_\ell^\top
 =U^\top C_\ell C_\ell^\top U.
\]
This is the matrix representation, in an orthonormal basis, of the
restriction of $C_\ell C_\ell^\top$ to $u^\perp$, so its eigenvalues are
exactly $\mu_{1,\ell},\ldots,\mu_{\ell-1,\ell}$.

Cauchy interlacing~\cite[Section~4.3]{HornJohnson2013}, applied to the
symmetric matrix $G_\ell$ and this principal submatrix, yields
\[
 \ell s_j(A)^2\geq\mu_{j,\ell}\geq\ell s_{j+1}(A)^2,
 \qquad 1\leq j\leq\ell-1.
\]
In particular, interlacing is applied to the Gram matrix after orthogonal
conjugation, not to the generally nonsymmetric matrix $A$.
Reindexing gives, for $2\leq k\leq\ell-1$,
\begin{equation*}
  \mu_{k,\ell}
  \leq \ell s_k(A)^2
  \leq \mu_{k-1,\ell}.
\end{equation*}
For fixed $k\geq2$, both endpoints converge in probability to $E_+$.  More generally, if $2\leq k\leq r_\ell$,
then interlacing and the ordering of the $\mu_{j,\ell}$ give
\[
 \mu_{r_\ell,\ell}
 \leq \ell s_k(A)^2\leq\mu_{1,\ell}.
\]
The uniform convergence of the compressed eigenvalues proves the uniform
claim and completes the proof of
\cref{thm:main}.

\section{Square model and branch selection}
\label{sec:square}

Take $\ell=p=d$ and $X=I_d$, so $\rho=1$.  The square-model analysis in
\cite[Appendix~C]{HayaseCollinsKarakida2026} uses
\begin{equation*}
  \Kbeta(w):=\frac1w+aw+\frac{bw}{1-bw^2},
  \qquad 0<w<b^{-1/2},
\end{equation*}
and identifies its squared upper bulk edge as follows
\cite[Eqs.~(C.2)--(C.3) and Remark~C.2]{HayaseCollinsKarakida2026}:
\begin{equation}
  E_+(\beta)=\Kbeta(w_*)^2,
  \qquad \Kbeta'(w_*)=0.
  \label{eq:square-reference-edge}
\end{equation}
The connection to \eqref{eq:E-Phi} and the branch choice are made
explicit below.

For $\rho=1$ and $0<y:=bw^2<1/2$, put
\begin{equation*}
  \eta=\frac{y}{1-y}.
\end{equation*}
The transformations in \eqref{eq:z-eta} and \eqref{eq:q-eta} become
\begin{equation*}
  z=\frac{\eta}{(1+\eta)^2}=y(1-y),
  \qquad
  q(\eta)=\frac{w}{\Kbeta(w)}.
\end{equation*}
Because $y<1/2$, this is the physical branch of the moment transform.
Direct substitution in \eqref{eq:Phi-eta} gives
\begin{equation*}
  \Phi_{\beta,1}\!\left(\frac{w}{\Kbeta(w)}\right)
  =\Kbeta(w)^2.
\end{equation*}

To see which algebraic branch is used, write $\kappa=a/b$.  The equation
$\Kbeta'(w_*)=0$ is equivalent to
\begin{equation*}
  \kappa y_*^3-2\kappa y_*^2+(\kappa+3)y_*-1=0,
  \qquad y_*=bw_*^2.
\end{equation*}
The left side equals $-1$ at $y=0$ and $4\kappa/27>0$ at $y=1/3$.
Since $\Kbeta$ is strictly convex, its critical point is unique, and hence
\begin{equation*}
  0<y_*<\frac13<\frac12.
\end{equation*}
For $\rho=1$,
\begin{equation*}
  M_1(z)=\frac{1-\sqrt{1-4z}}{2z},
  \qquad M_1(0)=1.
\end{equation*}
At $z=y_*(1-y_*)$, the bound $y_*<1/2$ forces
\[
  \sqrt{1-4y_*(1-y_*)}=1-2y_*,
\]
and therefore
\begin{equation*}
  M_1(y_*(1-y_*))=\frac1{1-y_*}.
\end{equation*}
The alternative value $1/y_*$ is the other algebraic branch and is
incompatible with $M_1(0)=1$.  This is precisely why $y_*$ selects the branch used
in \cite{HayaseCollinsKarakida2026} for \eqref{eq:square-reference-edge}.

\begin{corollary}[Square Gaussian attention]
\label[corollary]{cor:square}
If $\ell=p=d$, $X=I_d$, and $\beta>0$ is fixed, then for every fixed
$k\geq2$,
\begin{equation*}
  d s_k(A)^2\xrightarrow{\Pp}\Kbeta(w_*)^2.
\end{equation*}
In particular,
\begin{equation*}
  d s_2(A)^2\xrightarrow{\Pp}
  \max\supp\nu_\infty.
\end{equation*}
\end{corollary}

Numerical values for $\beta=1$, together with bounds on their numerical
error, are given in Appendix~\ref{app:numerical-edge}.

\appendix
\section{Numerical evaluation of the square-model edge}
\label{app:numerical-edge}

For $\beta=1$, $a=e-2$ and $b=1$.  The exact limiting constants in
Corollary~\ref{cor:square} are determined by the unique root
$y_*\in(0,1/3)$ of
\[
  (e-2)y(1-y)^2+3y-1=0,
\]
and by
\[
  E_+(1)=\frac{\bigl(1+(e-2)y_*+y_*/(1-y_*)\bigr)^2}{y_*}.
\]
Their numerical values, rounded to six decimal places, are
\[
  E_+(1)\approx9.009543,
  \qquad \sqrt{E_+(1)}\approx3.001590.
\]
Each displayed approximation has absolute error less than $5\times10^{-7}$.
These are deterministic numerical error bounds for the limiting constants;
they do not quantify finite-dimensional fluctuations of the random singular
values.

The error bounds can be checked using rational arithmetic.  Set
\[
  e_-:=\sum_{j=0}^{30}\frac1{j!},
  \qquad e_+:=e_-+\frac1{30\cdot30!},
\]
so that $e_-<e<e_+$.  Bisection, with the polynomial evaluated using these
bounds, brackets $y_*$ between
\[
  y_-:=\frac{298169093901984}{10^{15}},
  \qquad y_+:=\frac{298169093901985}{10^{15}}.
\]
Indeed, $(e_+-2)y_-(1-y_-)^2+3y_--1<0$ and
$(e_--2)y_+(1-y_+)^2+3y_+-1>0$.
Writing
\[
  B_\pm:=1+(e_\pm-2)y_\pm+\frac{y_\pm}{1-y_\pm},
\]
monotonicity of $1+(e-2)y+y/(1-y)$ in both $e$ and $y$ gives
\[
  \frac{B_-^2}{y_+}<E_+(1)<\frac{B_+^2}{y_-}.
\]
Comparing these rational bounds, and their positive square roots, with the
displayed decimal approximations gives the stated error bounds.

\section*{Acknowledgements}

I thank Tomohiro Hayase for helpful comments on the presentation.

ChatGPT Sol5.6 was used for calculations, proof ideas, and editorial
assistance. Some of its suggestions were helpful, while others were
misleading. I independently verified all computations and take full
responsibility for the contents of this paper.

\begingroup
\raggedright

\endgroup

\end{document}